\documentclass[12pt, a4paper, reqno]{amsart}

\usepackage{a4wide}
\usepackage{amscd, amsmath,amssymb, amsfonts}
\usepackage[utf8x]{inputenc}                
\usepackage{enumerate}

\newtheorem{thm}{Theorem}
\newtheorem{lem}[thm]{Lemma}
\newtheorem{cor}[thm]{Corollary}
\newtheorem{prop}[thm]{Proposition}

\newtheorem{obs}[thm]{Observation}

\theoremstyle{definition}

\newtheorem{ex}[thm]{Example}

\newcommand{\LieG}{{\mathsf G}}
\newcommand{\LieK}{{\mathsf K}}
\newcommand{\LieH}{{\mathsf H}}
\newcommand{\LieL}{{\mathsf L}}

\newcommand{\lieA}{{\mathfrak{a}}}

\newcommand{\lieG}{{\mathfrak{g}}}
\newcommand{\lieH}{{\mathfrak{h}}}
\newcommand{\lieK}{{\mathfrak{k}}}

\newcommand{\lieM}{{\mathfrak{m}}}
\newcommand{\lieP}{{\mathfrak{p}}}

\newcommand{\lieSO}{{\mathfrak{so}}}
\newcommand{\lieSU}{{\mathfrak{su}}}

\newcommand{\lieT}{{\mathfrak{t}}}
\newcommand{\fix}{{\mathfrak{fix}}}
\newcommand{\Root}{{\mathcal{R}}}

\newcommand{\Antipodal}{{\mathcal{A}}}

\newcommand{\Ad}{{\mathrm{Ad}}} 
\newcommand{\ad}{{\mathrm{ad}}}
\newcommand{\Exp}{{\mathrm{Exp}}}

\newcommand{\Aut}{{\mathrm{Aut}}} 
\newcommand{\Int}{{\mathrm{Int}}}

\newcommand{\Fix}{{\mathrm{Fix}}}
\newcommand{\reg}{{\mathrm{reg}}}

\def\N{{\mathbb N}}

\def\Z{{\mathbb Z}}
\def\R{{\mathbb R}}
\def\C{{\mathbb C}}

\title{Natural $\Gamma$-symmetric structures on $R$-spaces}

\author[Quast]{Peter Quast}
\author[Sakai]{Takashi Sakai}

\address[Quast]{Institut f\"{u}r Mathematik, Universit\"{a}t Augsburg, 
86135 Augsburg, Germany}
\email{quast@math.uni-augsburg.de}

\address[Sakai]{Department of Mathematical Sciences,
Tokyo Metropolitan University,
1-1 Min\-ami-Osawa, Hachioji-shi,
Tokyo 192-0397, Japan}
\email{sakai-t@tmu.ac.jp}


\thanks{The second author, Takashi Sakai, was supported by JSPS KAKENHI Grant Number 17K05223.}
\begin{document}

\begin{abstract}
We classify $R$-spaces that admit a certain natural $\Gamma$-symmetric structure. 
We further determine the maximal antipodal sets 
of these structures.
\end{abstract}

\maketitle
\section{Introduction}
In 1981 \textsc{Lutz} \cite{Lutz} introduced $\Gamma$-symmetric spaces, a  generalization of
$k$-symmetric spaces. 
In this article we consider natural $\Gamma$-symmetric 
structures on $R$-spaces using $\Gamma$-symmetric triples introduced by 
\textsc{Goze} and \textsc{Remm} \cite{GozeRemm}. In Theorem \ref{THM: Main}
$R$-spaces admitting a natural $\Gamma$-symmetric structure 
are characterized in terms of roots.
This allows us to classify these $R$-spaces in Section \ref{SEC: Classification}. 
For $\Gamma=\Z_2$ we recover the famous symmetric $R$-spaces of \textsc{Kobayashi}, \textsc{Nagano} \cite{KN} and 
\textsc{Takeuchi} \cite{Ta65}.  In Section \ref{SECT: Maximal antipodal sets} we further determine 
the  maximal antipodal sets of the natural $\Gamma$-symmetric 
structures and extend a result due to \textsc{Tanaka} and \textsc{Tasaki}  \cite{TT}.
We further describe the relation of these sets 
with the topology of the $R$-space generalizing a theorem of \textsc{Takeuchi} \cite{Ta89}.
Finally, in Section \ref{SECT: Maximality}, we show that natural $\Gamma$-symmetric structures 
are maximal
among certain types of  such structures.

\section{Preliminaries}

\subsection{$\Gamma$-symmetric triples}
\label{SECT: Gamma symmetric triples}
Following \textsc{Goze} and \textsc{Remm} \cite{GozeRemm},
a \emph{$\Gamma$-symmetric triple} is a triple $(\LieL,\LieH, \Gamma)$
consisting of a connected Lie group $\LieL,$ a closed subgroup $\LieH$ of $\LieL$ and a finite 
abelian subgroup $\Gamma$ of the automorphism group  $\Aut(\LieL)$ of $\LieL$ such that
\begin{equation}
 \label{EQ: Gamma sym triple}
 \Fix(\Gamma,\LieL)_0\subset \LieH\subset \Fix(\Gamma,\LieL).
\end{equation}
 Here 
 $\Fix(\Gamma,\LieL)=\{l\in\LieL\; |\; \gamma(l)=l\;\text{for all}\; \gamma\in \Gamma\}$ 
 and $\Fix(\Gamma,\LieL)_0$ denotes its identity component.
 If $\Gamma=\Z_2,$ then $(\LieL,\LieH)$ is a symmetric pair (see e.g.\ \cite[p.\ 209]{He}). 
  \textsc{Goze} and \textsc{Remm} \cite{GozeRemm} have shown that if  
  $(\LieL,\LieH, \Gamma)$  is a $\Gamma$-symmetric triple, then the coset space 
  $\LieL/\LieH$ carries a natural $\Gamma$-symmetric structure in the sense of \textsc{Lutz} \cite{Lutz}:
  For a point $x=l_x\LieH\in\LieL/\LieH,\; l_x\in\LieL,$ we set 
\begin{equation}
 \label{EQ: Lutz Gamma}
 \Gamma_{x}:=\{\gamma_x:\LieH/\LieL\to \LieH/\LieL\; |\; \gamma\in\Gamma\}, 
\end{equation}
with $\gamma_{x}(l\LieH)=l_x\gamma(l_x^{-1}l)\LieH$ for  $l\LieH\in\LieL/\LieH,\; l\in\LieL.$
Denoting by $L_l$ the left action of $l\in\LieL$ on the coset space $\LieL/\LieH$ 
we see that our $\Gamma$-symmetric structure on $\LieL/\LieH$ is $\LieL$-equivariant, since
\begin{equation}
\label{EQ: equivariant Gamma structure}
 \gamma_{x}=L_{l_x}\circ \gamma_{o}\circ L_{l_x^{-1}}
\end{equation}
for $x=l_x\LieH$ and $o=e\LieH.$

\subsection{Root systems of symmetric spaces}
In this section we merely recall well known facts about symmetric spaces and we introduce our notation. 
For further details, explanations and proofs on the well-established theory of 
symmetric spaces we refer to the standard literature \cite{He}, \cite{L-II} and \cite[Chap.\ 8]{Wolf}.
Let us fix an adjoint symmetric space $P$ of compact type,
that is $P$ is a symmetric space of compact type 
with the property that if $P$ finitely covers another symmetric space, 
 then $P$ is isomorphic to it. Let $\LieG$ denote the identity component of the isometry group of $P.$ 
We choose a base point $o\in P$ and we set $\LieK=\{g\in\LieG\; |\;  g(o)=o\}.$ 
The principal bundle 
$$\pi: \LieG\to P, \quad g\mapsto g(o)$$
gives an identification of 
$P$ with the coset space $\LieG/\LieK.$ Let $s_o$ be the geodesic symmetry of $P$ at $o$ and let 
$\sigma:\LieG\to\LieG,\; g\mapsto s_ogs_o.$ 
The Lie algebra $\lieK$ of $\LieK$ is the fixed point set of the 
differential $\sigma_*$ of $\sigma$ at the identity. Let $\lieP$ be the fixed space of $-\sigma_*.$
The map $\pi_*|_\lieP:\lieP\to T_oP,$ where $\pi_*$ denotes the differential of $\pi$ at the identity, is an isomorphism.
The Riemannian exponential map $\Exp_o:T_oP\to P$ satisfies 
$\pi\circ \exp|_\lieP=\Exp_o\circ \pi_*|_\lieP,$ 
where $\exp:\lieG\to\LieG$ is the Lie theoretic exponential map. Using the
identification by $\pi_*|_\lieP$ the adjoint action of $\LieK$ on $\lieP$ coincides with the linear isotropy action 
of $\LieK$ on $T_oP.$\par
Let $\lieA$ be a maximal abelian subspace of $\lieP.$ The dimension of $\lieA$ is called the \emph{rank} of $P$ and is
usually denoted by $r.$ Let $\lieA^*$ be the dual vector space of $\lieA.$ For each $\alpha\in\lieA^*$ we set
$$\lieG_\alpha:=\{Y\in\lieG\otimes \C\; |\;  [H,Y]=\sqrt{-1}\alpha(H)Y\; \text{for all}\; 
H\in\lieA\}.$$
The root system of $P$ is 
$$\Root:=\{\alpha\in\lieA^*\; |\; \alpha\neq 0\; \text{and}\; \lieG_{\alpha}\neq\{0\}\}.$$
Let $\Sigma:=\{\alpha_1,\dots,\alpha_r\}$ be a system of simple roots of $\Root$ 
and let $\Root^+\subset\Root$ be the corresponding set of positive roots. Then $\Sigma$ is 
a basis of $\lieA^*$ with the property that every root $\alpha\in\Root^+$ can be 
written as 
\begin{equation}
 \label{EQ: positive root as sum}
 \alpha=\sum\limits_{j=1}^rc^\alpha_j \alpha_j \qquad\text{with}\; c^\alpha_j\in\N\cup\{0\}\; 
 \text{for all}\; j\in I_{\reg}:=\{1,\dots ,r\}.
\end{equation}
We get the following  direct sum decompositions:
\begin{eqnarray*}
 \lieP&=&\lieA\oplus\sum\limits_{\alpha\in\Root^+}\lieP_\alpha \quad \text{with}\quad \lieP_\alpha=(\lieG_\alpha\oplus\lieG_{-\alpha})\cap\lieP,\\
 \lieK&=&\lieM\oplus\sum\limits_{\alpha\in\Root^+}\lieK_\alpha \quad \text{with}\quad
 \lieK_\alpha=(\lieG_\alpha\oplus\lieG_{-\alpha})\cap\lieK\quad \text{and}\quad
 \lieM=\{Y\in\lieK\; |\;  [\lieA,Y]=\{0\}\}.
\end{eqnarray*}
Let $\{\xi_1,\dots,\xi_r\}$ be the basis of $\lieA$  dual to $\Sigma,$ that is 
$\alpha_j(\xi_k)=0$ if $j\neq k$ and $\alpha_j(\xi_j)=1.$ Then we have for each $\alpha\in\Root^+$ and 
each $j\in I_{\reg}:=\{1,\dots, r\}$
by
\eqref{EQ: positive root as sum}: 
\begin{equation}
 \label{EQ: coefficient}
 \alpha(\xi_j)=c^{\alpha}_j.
\end{equation}\par

It is well known (see e.g.\ \cite[pp.\ 25, 69]{L-II}) that the unit lattice 
$\Lambda=\{H\in\lieA\; |\; \exp(H)(o)=o\}=\{H\in\lieA\; |\; \exp(H)\in\LieK\}$ 
of the adjoint symmetric space $P$ of compact type 
is given by 
$$\Lambda=\{X\in\lieA\; |\; \alpha(X)\in\pi\Z\;\; \text{for all}\; \alpha\in\Root\}=\mathrm{span}_{\pi\Z}(\xi_1,\dots,\xi_r).$$
In particular the elements 
\begin{equation}
\label{EQ: k-j}
 k_j:=\exp(\pi\xi_j),\quad j\in\{1,\dots, r\},
\end{equation}
lie in $\LieK$ and we have
\begin{equation}
\label{EQ: gamma_j}
 \gamma^j:=\Int(k_j)\in\Aut(\LieK_0),
\end{equation}
where $\LieK_0$ denotes the identity component of $\LieK.$ Note that 
$\gamma^j\gamma^j=e$ for all $j\in\{1,\dots, r\},$ since $\Ad_\LieG(\exp(2
\pi\xi_j))$ is the identity on $\lieG.$

\subsection{$R$-spaces}
An \emph{$R$-space} is an orbit $\Ad_\LieG(\LieK_0)\xi\subset\lieP$ for some $\xi\in\lieP\setminus\{0\}.$
In the literature $R$-spaces are sometimes called \emph{orbits of $s$-representations}.
If $P$ is a compact Lie group, these orbits are also known as generalized (complex) flag manifolds. 
Every $R$-space in $\lieP$ is 
equivariantly isomorphic to the adjoint orbit of a
\emph{canonical element} 
in the sense of \textsc{Burstall} and \textsc{Rawnsley} \cite{BR}. This is an element of the form 
$$\xi_I:=\sum\limits_{i\in I} \xi_i,$$
where $I$ is a non-empty subset of $I_{\reg}=\{1,\dots,r\}.$
We denote the corresponding $R$-space by 
$$X_I:=\Ad_{\LieG}(\LieK_0)\xi_I\subset\lieP.$$
As a homogeneous space, $X_I$  can be identified with the coset space 
$\LieK_0/\LieH_I$ where
$$\LieH_I=\{k\in\LieK_0\; |\;  \Ad_\LieG(k)\xi_I=\xi_I\}\subset\LieK_0.$$
\par

\begin{lem}
\label{LEM: inclusion isotropy}
 Let $\emptyset\neq\tilde{I}\subset I\subset I_{\reg}=\{1,\dots, r\},$ then $\LieH_I\subset \LieH_{\tilde{I}}.$
\end{lem}

\begin{proof}
 The euclidean vector space $\lieP\cong T_oP$ carries an orthogonal
 Lie triple structure given by the curvature tensor of $P$ at $o$
 which can be expressed by double Lie brackets. Using the canonical isomorphism  $\pi_*|_\lieP:\lieP\to T_oP$ 
we can identify $\LieK_0$ 
with a subgroup of the group of orthogonal 
 Lie triple automorphisms of $\lieP$  via the linear isotropy representation.  If $P$ is simply connected, 
 then
$\LieK_0$ is actually the identity component of the group of orthogonal 
 Lie triple automorphisms of $\lieP$  (see e.g.\ proof of \cite[Thm.\ 8.4.7]{Wolf}). 
 Actually the same holds true if $P$ is not simply connected, since the Lie algebra $\lieK$
 of $\LieK_0$ coincides with the Lie algebra of the isotropy group of the universal covering of $P$ 
 at a point in a fiber over $o.$\par
 Let $\lieP^*$ denote the euclidean vector space $\lieP$ endowed with 
 the Lie triple structure that differs from the one of $\lieP$ by a sign. 
 Then 
 $\LieK_0$ is again the identity component of the 
 group of orthogonal 
 Lie triple automorphisms of $\lieP^*.$ The symmetric space $P^*$ corresponding to $\lieP^*$ 
 is the non-compact dual symmetric space of $P,$ and $\lieP^*$ is canonically isomorphic
to the tangent space 
 of $P^*$ at a certain point $o^*\in P^*.$ Since $P^*$ is simply connected, $\LieK_0$ can be identified with 
 the isotropy group of $o^*$ within the identity component $\LieG^*$ of the isometry group of $P^*.$\par
We now consider $X_I=\Ad_\LieG(\LieK_0)\xi_I$ as an isotropy orbit of $P^*$ lying in 
$\lieP^*\cong T_{o^*}P^*.$ By looking at the corresponding geodesic rays, 
the isotropy orbit $X_I\subset\lieP^*$ induces a $\LieK_0$ orbit 
$X_I(\infty)$
in the boundary  at infinity  $P^*(\infty)$ of $P^*.$ The $\LieK_0$-orbit at infinity $X_I(\infty)$
is also a $\LieG^*$-orbit at infinity (see e.g.\ \cite[Cor.\ 2.17.2]{Eberlein}). The isotropy group 
$\LieG^*_{\xi_I(\infty)}\subset\LieG^*$ of the point 
$\xi_I(\infty)\in P^*(\infty)$ 
corresponding to $\xi_I$ is a parabolic subgroup of $\LieG^*.$ Our claim 
follows from the inclusions of the corresponding parabolic subgroups (see \cite[p.\ 249]{BGS}, \cite[Sect.\ 2.17]{Eberlein})
since the isotropy of $\xi_I(\infty)$ within $\LieK_0$ coincides with $\LieH_I$ and 
$\LieH_I=\LieG^*_{\xi_I(\infty)}\cap\LieK_0.$
\end{proof}

\section{Natural $\Gamma$-symmetric structures on $R$-spaces}
\subsection{Admissible subsets}
Let $I$ be a non-empty subset of $I_{\reg}=\{1,\dots, r\}$ 
and let $X_I\cong \LieK_0/\LieH_I.$  We define
$$
\Gamma^I=\langle \gamma^i:\; i\in I\rangle\subset\Aut(\LieK_0)$$
as the group generated by the elements $\gamma^i$ given in \eqref{EQ: gamma_j} with $i\in I.$
As a group, $\Gamma^I$ is isomorphic to $(\Z_2)^{|I|}$ where $|I|$ is the cardinality of 
$I.$ A non-empty subset $I$ of $I_{\reg}=\{1,\dots, r\}$  is called \emph{admissible}
if 
$(\LieK_0,\LieH_I,\Gamma^I)$ is a $\Gamma^I$-symmetric triple, that is by \eqref{EQ: Gamma sym triple} if 
 \begin{equation}
 \label{EQ: Gamma sym pair}
  \Fix(\Gamma^I,\LieK_0)_0\subset \LieH_I\subset \Fix(\Gamma^I,\LieK_0).
 \end{equation}

\begin{obs}\label{OBS: right hand inclusion}
 The right hand side inclusion in 
 \eqref{EQ: Gamma sym pair}, $\LieH_I\subset \Fix(\Gamma^I,\LieK_0),$ is always true.
\end{obs}
\begin{proof}
Let $k\in\LieH_I$ and let $i\in I$ then we have
$\gamma^i(k)k^{-1}=\exp(\pi\xi_i)k\exp(-\pi\xi_i)k^{-1}
=\exp(\pi\xi_i)\exp(-\pi\Ad_\LieG(k)\xi_i)
=\exp(\pi\xi_i)\exp(-\pi\xi_i)=e,$ 
 since 
$\LieH_I\subset \LieH_{\{i\}}$ by Lemma \ref{LEM: inclusion isotropy}. Thus 
$\gamma^i(k)=k.$
\end{proof}

Now we only have to deal with the left hand side inclusion in \eqref{EQ: Gamma sym pair},
$
 \Fix(\Gamma^I,\LieK_0)_0\subset \LieH_I,
$
which in view of Observation \ref{OBS: right hand inclusion} is equivalent to 
$$\lieH_I=\fix(\Gamma^I,\LieK_0),$$ 
where  $\lieH_I$ denotes the Lie algebra of $\LieH_I$ and $\fix(\Gamma^I,\LieK_0)$  the Lie algebra of $\Fix(\Gamma^I,\LieK_0).$ 
We have
$$\lieH_I=\{X\in\lieK\; |\; [X,\xi_I]=0\}=\lieM\oplus \sum\limits_{\alpha\in  \Root^0_I}\lieK_\alpha$$
 with 
\begin{equation}
\label{EQ: Root-0-I}
  \Root^0_I:=\{\alpha\in\Root^+:\; \alpha(\xi_I)=0\}
  =\{\alpha\in\Root^+\; |\;  \alpha(\xi_i)=0\; \text{for all}\;i\in I\},
\end{equation}
and 
\begin{eqnarray*}
 \fix(\Gamma^I, \LieK_0)&=& \{X\in\lieK\; |\; \forall t\in\R\;\; \forall i\in I:
 \gamma^i(\exp(tX))=\exp(tX)\}\\
 &=& \{X\in\lieK\; |\; \forall t\in\R\;\; \forall i\in I:
 \Int(k_i)(\exp(tX))=\exp(t\Ad_\LieG(k_i)X)=\exp(tX)\}\\
 &=&\{X\in\lieK\; |\; \forall i\in I:
 \Ad_\LieG(k_i)X=X\}
\; =\;  \{X\in\lieK\; |\; \forall i\in I:
 e^{\pi\ad(\xi_i)}X=X\}
 \\
 &=& \lieM\oplus \sum\limits_{\alpha\in\Root^{\mathrm{ev}}_I}\lieK_\alpha\\
\end{eqnarray*}
with
$$\Root^{\mathrm{ev}}_I:= \{\alpha\in\Root^+\; |\;  \alpha(\xi_i)\;\text{is even for all}\; i\in I  \}.$$
Summing up we have shown:

\begin{thm}[Characterization of admissible subsets in terms of $\Root$]
\label{THM: Main}
A non-empty subset  $I$ of $I_{\reg}=\{1,\dots, r\}$ is admissible if and only if 
for all $\alpha\in\Root^+$ we have
$$\big(\forall i\in I:\; c^\alpha_i\; \text{even}\big)\implies\big(\forall i\in I: c^\alpha_i=0\big).$$
\end{thm}

From Theorem \ref{THM: Main} we see:

\begin{cor}
Unions of admissible subsets of $I_\reg$ are admissible.
\end{cor}

\subsection{Principal orbits}
A root system $\Root$ is called \emph{reduced}
if for any $\alpha\in\Root$ we have $\Z\alpha\cap \Root=\{-\alpha,\alpha\}.$
Details on root systems can be found in \cite[Chap.\ V]{Se} and \cite[Chap.\ VI]{Bourbaki}.

\begin{prop}
Let $\Root$ be reduced, then $I_\reg=\{1,\dots, r\}$ is admissible.
\end{prop}

\begin{proof}
Let $\alpha=\sum\limits_{j=1}^rc^\alpha_j\alpha_j$ be a positive root. By Lemma \ref{LEM: Reduced root system} there 
exists $k\in\{1,\dots ,r\}$ such that $c^\alpha_k$ is odd. The claim now follows from Theorem \ref{THM: Main}.
\end{proof}

\begin{lem}
\label{LEM: Reduced root system}
 Let $\Root$ be a reduced root system, let $\Sigma=\{\alpha_1,\dots,\alpha_r\}$ be a system of simple roots of $\Root$ and 
let $\alpha=\sum\limits_{j=1}^r c^\alpha_j\alpha_j$ be an element of $\Root.$ Then there exists $j\in\{1,\dots, r\}$ such that 
 $c^\alpha_j$ is odd.
\end{lem}

\begin{proof}
One can check this statement by inspecting all irreducible reduced root systems (see e.g.\
 \cite[Planche I--IX, pp.\ 250--275]{Bourbaki}), but we prefer a conceptional argument.\par
 Let $\check{\Root}$ denote the system of inverse roots of $\Root$ which is also reduced. 
 Then $\Root$ is the inverse root system of 
 $\check{\Root},$ that is $\Root=\check{\check{\Root}}.$ Let $\check{\lieG}$ denote the compact semi-simple 
 real Lie algebra with root system 
 $\check{\Root}$ and let $\check{\lieT}$ be a maximal abelian subspace 
 of $\check{\lieG}.$ Then $\check{\Root}$ is isomorphic to the root system of $\check{\lieG}$ corresponding to 
$\check{\lieT}$ and $\Root$ is isomorphic to its inverse root system. The unit lattice $\check{\Lambda}\subset\check{\lieT}$ 
of the simply connected compact real Lie group $\check{\LieG}$ with Lie algebra 
$\check{\lieG}$
is (see e.g.\ \cite[p.\ 25]{L-II})
$$\check{\Lambda}=\mathrm{span}_{2\pi\Z}(\check{\check{\Root}})=\mathrm{span}_{2\pi\Z}(\Root)
=\mathrm{span}_{2\pi\Z}(\alpha_1,\dots,\alpha_r).$$
Assume now that there exists a root $\alpha=\sum\limits_{j=1}^r c^\alpha_j\alpha_j\in\Root$ with $c^\alpha_j$ even for 
all $j\in\{1,\dots, r\}.$ Then
$\pi\alpha$ lies in $\check{\Lambda}.$ But, since $\Root$ is reduced, we can  find a simple root system 
$\{\beta_1,\dots, \beta_r\}$ of $\Root$ with $\beta_1=\alpha$ (see e.g.\ \cite[Prop.\ 15, p.\ 154]{Bourbaki}).  Then 
$\pi\beta_1=\pi\alpha\in \check{\Lambda}=\mathrm{span}_{2\pi\Z}(\Root)=\mathrm{span}_{2\pi\Z}(\beta_1,\dots,\beta_r),$
 a contradiction.
\end{proof}


\subsection{Orbits of sums of extrinsic symmetric elements}

We call an element $\xi_j,\; j\in\{1,\dots, r\},$ \emph{extrinsic symmetric}, if 
$\alpha_j$ has coefficient one in the highest root. 

\begin{prop}
\label{PROP: sum extr sym}
If for all $i\in I$ the element $\xi_i$ is extrinsic symmetric, that is 
  $\xi_I$ is the sum of extrinsic symmetric elements, then
$I$ is admissible.
\end{prop}

\begin{proof}
Let
  $\alpha=\sum\limits_{j=1}^rc^\alpha_j\alpha_j$ be a positive root.
  Since for all $i\in I$ the simple root
  $\alpha_i$ has coefficient 1 in the highest root, we have 
  $c^\alpha_i\in\{0,1\}$ for all $i\in I.$   Our claim now follows from Theorem \ref{THM: Main}.
\end{proof}

If $I=\{j\}$ and  $\xi_I=\xi_j$ is extrinsic symmetric, then 
$X_I$ is a symmetric $R$-space in the sense of \textsc{Takeuchi} \cite{Ta65} and \textsc{Kobayashi} 
and \textsc{Nagano} \cite{KN} (see also \cite[Lemma 2.1]{MQ}).

\section{Classification}\label{SEC: Classification}
Theorem \ref{THM: Main} allows us
to classify all admissible subsets $ I$ of $I_{\reg}$
in terms of the root system $\Root$ of $P,$ which we may assume to be irreducible.
We proceed case-by-case by the type of $\Root.$
The description of the irreducible reduced root systems is taken from \cite[Planche I--IX, pp.\ 250--275]{Bourbaki}.

\subsection{Type $A_r$}
Since each simple root has coefficient 1 in the highest root, every non-empty subset
$I$ of  $I_{\reg}$ is admissible
by Proposition \ref{PROP: sum extr sym}.

\subsection{Type $B_r,\; r\geq 2$}
From \cite[Planche II, pp.\ 252--253]{Bourbaki} we see that 
$\Root^+=\{e_j\; |\;  1\leq j\leq r\}\cup\{e_j\pm e_k\; |\;  1\leq j<k\leq r\}$ with simple roots
$\Sigma=\{\alpha_1=e_1-e_2, \, \dots,\,  \alpha_{r-1}=e_{r-1}-e_r,\, \alpha_r=e_r\}.$
We have 
\begin{itemize}
 \item $e_j=\sum\limits_{k=j}^r \alpha_k=\alpha_j+\dots +\alpha_r$ for $1\leq j\leq r.$
 \item $e_j-e_k=\sum\limits_{l=j}^{k-1} \alpha_l=\alpha_j+\dots+\alpha_{k-1}$ for  $1\leq j<k\leq r.$
 \item $e_j+e_k=\sum\limits_{l=j}^{k-1} \alpha_l+2\sum\limits_{l=k}^{r} \alpha_l
 =\alpha_j+\dots+\alpha_{k-1}+2\alpha_k+\dots+2\alpha_r$ for  $1\leq j<k\leq r.$
\end{itemize}
Looking at the highest root $\delta=e_1+e_2=\alpha_1+2\alpha_2+\dots +2\alpha_r$ 
one sees that every admissible subset $I$
must contain 1. Moreover, take $k\in\{1,\dots, r-1\}$ and look at the root 
$e_k+e_{k+1}=\alpha_k+2\alpha_{k+1}+\dots +2\alpha_r$
we see that if $k\notin I$ but $I\cap\{k+1,\dots, r\}\neq \emptyset,$ then 
$I$ is not admissible. 
Thus the possible candidates for admissible subsets $I$ are of the form
$$I=\{1,2,...,k\}$$
for some $k\in\{1,\dots, r\}.$ One easily checks that they  are indeed admissible.

\subsection{Type $C_r,\; r\geq 2$}
From \cite[Planche III, pp.\ 254--255]{Bourbaki} we see that 
$\Root^+=\{2e_j\; |\;  1\leq j\leq r\}\cup\{e_j\pm e_k\; |\;  1\leq j<k\leq r\}$ with simple roots
$\Sigma=\{\alpha_1=e_1-e_2, \, \dots,\,  \alpha_{r-1}=e_{r-1}-e_r,\, \alpha_r=2e_r\}.$
We have 
\begin{itemize}
 \item $2e_j=2\sum\limits_{k=j}^{r-1} \alpha_k+\alpha_r=2\alpha_j+\dots +2\alpha_{r-1}+\alpha_r$ 
 for $1\leq j\leq r.$
 \item $e_j-e_k=\sum\limits_{l=j}^{k-1} \alpha_l=\alpha_j+\dots+\alpha_{k-1}$ for  $1\leq j<k\leq r.$
 \item $e_j+e_k=\sum\limits_{l=j}^{k-1} \alpha_l+2\sum\limits_{l=k}^{r-1} \alpha_l+\alpha_r$ for  $1\leq j<k\leq r.$
\end{itemize}
Looking at the highest root $\delta=2e_1=2\alpha_1+\dots +2\alpha_{r-1}+\alpha_r$ 
one sees that every admissible subset $I$
must contain $r.$ One easily checks that $I$ is admissible if and only if $r\in I.$

\subsection{Type $D_r,\; r\geq 4$}
From \cite[Planche IV, pp.\ 256--257]{Bourbaki} we see that 
$\Root^+=\{e_j\pm e_k\; |\;  1\leq j<k\leq r\}$ with simple roots
$\Sigma=\{\alpha_1=e_1-e_2, \, \dots,\,  \alpha_{r-1}=e_{r-1}-e_r,\, \alpha_r=e_{r-1}+e_r\}.$
We have 
\begin{itemize}
 \item $e_j-e_k=\sum\limits_{l=j}^{k-1} \alpha_l$ for  $1\leq j<k\leq r.$
 \item $e_j+e_k=\sum\limits_{l=j}^{k-1} \alpha_l+2\sum\limits_{l=k}^{r-2} \alpha_l
 +\alpha_{r-1}+\alpha_r$ for  $1\leq j<k\leq r-1.$
  \item $e_j+e_r=\sum\limits_{l=j}^{r-2} \alpha_l+\alpha_r$ for  $1\leq j\leq r-1.$
\end{itemize}
Looking at the highest root 
$\delta=e_1+e_2=\alpha_1+2\alpha_2+\dots +2\alpha_{r-2}+\alpha_{r-1}+\alpha_r$ 
one sees that subsets of $\{2,\dots, r-2\}$ are not admissible.
Therefore $I\cap\{1,r-1,r\}$ is non-empty if $I$ is admissible.
One easily checks that if $I\cap\{r-1,r\}$ is non-empty, then $I$ is admissible.\par
We are left with the cases where $1\in I\subset\{1,\dots, r-2\}.$ 
Take $k\in\{1,\dots, r-3\}$ and look at the root 
$e_k+e_{k+1}=\alpha_k+2\alpha_{k+1}+\dots +2\alpha_{r-2}+\alpha_{r-1}+\alpha_r$
we see that if $k\notin I\subset\{1,\dots, r-2\}$ but $I\cap\{k+1,\dots, r-2\}\neq \emptyset,$ then 
$I$ is not admissible. 
Thus the possible candidates for admissible subsets $I$ with $1\in I\subset\{1,\dots, r-2\}$ are of the form
$$I=\{1,2,...,k\}$$
for some $k\in\{1,\dots, r-2\}.$ One easily checks that they are indeed  admissible.

\subsection{Type $E_r,\; r\in\{6,7,8\}$}
The simple roots, their numbering and the coefficients of the positive roots are taken from
\cite[Planche V-VII, pp.\ 260--270]{Bourbaki}. We observe:
\begin{enumerate}[(i)]
\item The positive root $\alpha_1+\alpha_2+\alpha_r+2\sum\limits_{j=3}^{r-1}\alpha_j$ 
 shows that if $I$ is admissible for $E_r, \; r\in\{6,7,8\},$ then
$I\cap \{1,2,r\}\neq\emptyset.$  
\item Positive roots of the root system of type $E_r,\; r\in\{6,7\},$ 
induce positive roots of $E_{r+1}$ whose coefficients in 
the simple root
$\alpha_{r+1}$ vanish and vice-versa. Therefore if $I\subset\{1,2,\dots, r+1\}$ is admissible for the type $E_{r+1}$ 
and $I\setminus\{r+1\}$ is non-empty, then
$I\setminus\{r+1\}$ must be admissible for the type $E_r.$
\item Positive roots whose coefficients in the simple roots are elements of $\{0,1\}$  do not give any restriction for admissible 
 subsets.
 
\end{enumerate}

\subsubsection{Type $E_6$}
Inspecting the coefficients of the positive roots  in the case $E_6$ we see:
\begin{itemize}
 \item If $I$ contains $1,$ then $I$ is admissible unless $I$ is equal to $\{1,4\},\; \{1,5\},\; \{1,4,5\}$ or $\{1,4,6\}.$
 \item If $I$ contains $2,$ then $I$ is admissible unless $I$ is a subset of $\{2,3,5\}.$ 
 \item If $I$ contains $6,$ then $I$ is admissible unless $I$ is equal to $\{3,6\},\; \{4,6\},\; \{1,4,6\}$ or $\{3,4,6\}.$
\end{itemize}

\subsubsection{Type $E_7$}
Inspecting the coefficients of the positive roots  in the case $E_7$ and the admissible 
sets for $E_6$ we see:
\begin{itemize}
\item Since $c^\delta_7=1$ for highest root $\delta$  we see that $I=\{7\}$ is admissible. Moreover, 
if $I\setminus\{7\}$ is non-empty, then $I$ is admissible for $E_7$ if and only if 
$I\setminus\{7\}$ is admissible for $E_6.$
\item If $I$ contains $1$ but not $7,$ then $I$ is admissible unless 
$I$ is a subset of $\{1,2,4,6\}$ or $\{1,4,5,6\}.$
 \item If $I$ contains $2$ and $I\cap\{1,7\}=\emptyset,$ then $I$ is admissible unless
$I$ is a subset of $\{2,3,4,6\}$ or $\{2,3,5,6\}.$ 
\end{itemize}

\subsubsection{Type $E_8$}
Inspecting the coefficients of the positive roots  in the case $E_8$ and the admissible 
sets for $E_7$ we see:
\begin{itemize}
 \item  If $I$ contains $8,$ then $I$ is admissible unless $I$ is a subset of $\{1,3,4,6,8\}$	
 or $I\setminus\{8\}$ is not admissible for $E_7.$
 \item If $I$ contains  $1$ but not $8,$ then $I$ is admissible unless  
 $I$ is a subset of $\{1,2,3,5,7\},$   $\{1,2,4,5,7\},$
      $\{1,2,4,6,7\},$ $\{1,3,4,5,7\},$ $\{1,3,4,6,7\}$ or $\{1,4,5,6,7\}.$ 
 \item If $I$ contains  $2$ and $\{1,8\}\cap I=\emptyset,$ then $I$ is admissible
 unless $I$ is a subset of 
 $\{2,3,4,5,7\},$ $\{2,3,4,6,7\}$ or $\{2,3,5,6,7\}.$
 \end{itemize}

\subsection{Type $F_4$}
The simple roots, their numbering and the coefficients of the positive roots are taken from
\cite[Planche VIII, pp.\ 272--273]{Bourbaki}. 
The highest root $2\alpha_1+3\alpha_2+4\alpha_3+2\alpha_4$
and the root $\alpha_1+2\alpha_2+2\alpha_3+2\alpha_4$ show
that every admissible subset $I$ must contain $\{1,2\}.$ 
Moreover
all positive roots $\alpha$ having at least one coefficient $\geq 2$
satisfy $c^\alpha_1$ odd or $c^\alpha_2$ odd. Thus every subset $\{1,2\}\subset I\subset 
\{1,2,3,4\}$ is admissible. 

\subsection{Type $G_2$}
The simple roots, their numbering and the coefficients of the positive roots are taken from
\cite[Planche IX, pp.\ 274--275]{Bourbaki}.
We have
$\Root^+=\{\alpha_1,  \, \alpha_2, \,  \alpha_1+\alpha_2,  \, 
2\alpha_1+\alpha_2,  \,  3\alpha_1+\alpha_2,  \, 3\alpha_1+2\alpha_2\}.$ 
Thus only $I=I_{\reg}=\{1,2\}$ is admissible.

\subsection{Type $BC_r$}
Since each simple root has coefficient 2 in the highest root  (see \cite[p.\ 476]{He}), there are 
\emph{no}  admissible subsets in this case.

\section{Maximal antipodal sets}\label{SECT: Maximal antipodal sets}

Let  $I$ be an admissible subset of $\{1,\dots, r\},$ that is 
$(\LieK_0,\LieH_I,\Gamma^I)$ is a $\Gamma^I$-symmetric triple and 
$\LieK_0/\LieH_I$ is a $\Gamma^I$-symmetric space by  \eqref{EQ: Lutz Gamma}.
Using the equivariant identification
$$
 \iota: \LieK_0/\LieH_I\to X_I:=\Ad_\LieG(\LieK_0)\xi_I,\quad 
 k\LieH_I\mapsto \Ad_\LieG(k)\xi_I
$$
we transfer the $\LieK_0$-equivariant $\Gamma^I$-symmetric structure from $\LieK_0/\LieH_I$ to $X_I$
by 
$$\gamma_{\iota(x)}(\iota(y)):=\iota(\gamma_x(y))$$
for all $x,y\in \LieK_0/\LieH_I$ and for all $\gamma\in \Gamma^I.$
In particular we have for $k\in\LieK$ and $i\in I$ using \eqref{EQ: k-j}
$$\gamma^i_{\xi_I}(\Ad_\LieG(k)\xi_I)=
\Ad_\LieG(\gamma^i(k))\xi_I=
\Ad_\LieG(k_ikk_i^{-1})\xi_I=\Ad_\LieG(k_i)\Ad_\LieG(k)\xi_I,$$
since 
$\Ad_\LieG(k_i^{-1})\xi_I=\Ad_\LieG(\exp(-\pi\xi_i))\xi_I=e^{-\pi\ad(\xi_i)}\xi_I=\xi_I$ as 
$[\xi_i,\xi_I]=0.$
We observe  that the map $\gamma^i_{\xi_I}$
extends  to the linear endomorphism $\Ad_\LieG(k_i)|_\lieP$ of $\lieP.$ 
With $$\Gamma^I_{\xi_I}:=\{\gamma_{\xi_I}\; |\; \gamma\in\Gamma^I\}=\langle \gamma^i_{\xi_I}:
i\in I\rangle\subset 
\mathrm{Diff}(X_I)$$ and 
$$\Gamma^I_{\xi_I,\lieP}:=\langle \Ad_\LieG(k_i)|_\lieP: i\in I\rangle\subset 
\mathrm{End}(\lieP),$$ the group generated by the elements 
$\Ad_\LieG(k_i)|_\lieP$ with $i\in I,$
we get by Theorem \ref{THM: Main} 
\begin{eqnarray*}
 \Fix(\Gamma^I_{\xi_I,\lieP}, \lieP)&=&\{Y\in \lieP\; |\;   \Ad_\LieG(k_i)(Y)=Y\; \text{for all}\; 
i\in I\}\\
&=& \lieA\oplus \sum\limits_{\substack{\alpha\in\Root^+\\ \forall i\in I:\; 2|\alpha(\xi_i)}}\lieP_\alpha
\;=\; \lieA\oplus \sum\limits_{\substack{\alpha\in\Root^+\\ \forall i\in I:\; \alpha(\xi_i)=0}}\lieP_\alpha
\;=\; \lieA\oplus \sum\limits_{\substack{\alpha\in\Root^+\\  \alpha(\xi_I)=0}}\lieP_\alpha\\
&=& \{Y\in\lieP\; |\;  [Y,\xi_I]=0\}
\end{eqnarray*}
since $I$ is admissible. Moreover
$$\Fix(\Gamma^I_{\xi_I}, X_I)=\{y\in X_I\; |\;  \gamma(y)=y\; \text{for all}\; 
\gamma\in \Gamma^I_{\xi_I}\}=\Fix(\Gamma^I_{\xi_I,\lieP}, \lieP)\cap X_I.$$
Since the $\Gamma^I$-symmetric structure is $\LieK_0$-equivariant, we have 
\begin{equation}
\label{EQ: abelian 2}
\begin{array}{lll}
\Fix(\Gamma^I_{\Ad_\LieG(k)\xi_I}, X_I)
&=&\Ad_\LieG(k)\big(\Fix(\Gamma^I_{\xi_I}, X_I)\big)\\
&=&\Ad_\LieG(k)\big( \{Y\in\lieP\; |\;  [Y,\xi_I]=0\}\big)\cap X_I\\
&=&\{Z\in\lieP\; |\;  [Z,\Ad_\LieG(k)\xi_I]=0\}\cap X_I.
 \end{array}
\end{equation}

For Riemannian symmetric spaces the notion of an antipodal set has been introduced by \textsc{Chen} and \textsc{Nagano}
\cite{CN}. This notion naturally extends to $\Gamma$-symmetric spaces. 
A subset $\Antipodal$ of $X_I$ is called \emph{antipodal} (for $\Gamma^I$), if
\begin{equation}
\label{EQ: Antipodal}
  \forall a,\tilde{a}\in \Antipodal \quad \forall \gamma\in\Gamma^I: \quad
 \gamma_a(\tilde{a})=\tilde{a}.
\end{equation}
An antipodal subset $\Antipodal$ of $X_I$ is called \emph{maximal}, if $\Antipodal$ is not a proper subset of another 
antipodal subset of $X_I.$ 
Let $\Antipodal$ be a maximal antipodal subset of $X_I\subset\lieP$ with $\xi_I\in\Antipodal.$ Then the linear subspace
$\mathrm{span}_\R(\Antipodal)$ of $\lieP$
is abelian by \eqref{EQ: abelian 2}. Let $\tilde{\lieA}$ be a maximal abelian subspace of $\lieP$ containing 
$\mathrm{span}_\R(\Antipodal).$ Since $X_I\cap \tilde{\lieA}$ is an antipodal subset of $X_I$ by \eqref{EQ: abelian 2}, 
we have 
$$\Antipodal =X_I\cap \tilde{\lieA}$$
by maximality. It is well known that the intersection of an orbit of the identity component of the isotropy group 
of a Riemannian symmetric space $P$ of compact type with a maximal abelian subspace $\tilde{\lieA}\subset\lieP$ is the orbit 
of the Weyl group $W(P,\tilde{\lieA})$ of $P$ corresponding to $\tilde{\lieA}.$ Thus we have 
$$\Antipodal =W(P,\tilde{\lieA})\xi_I.$$
Since the $\Gamma^I$-symmetric structure on $X_I$ is $\LieK_0$-equivariant and since $\LieK_0$ acts transitively on the set 
of maximal abelian subspaces of $\lieP,$ we can summarize:
\begin{thm}
\label{THM: antipodal}
 Every maximal antipodal subset of $X_I$ is of the form 
 $X_I\cap \lieA'$ for some maximal abelian subspace $\lieA'$ in $\lieP$ and therefore 
 an orbit of the corresponding Weyl group $W(P,\lieA').$ Moreover, any two maximal 
 antipodal subsets of $X_I$ are conjugate by an element of $\LieK_0.$
\end{thm}
Theorem \ref{THM: antipodal} generalizes the result \cite[Theorem 4.3]{TT} of \textsc{Tanaka} and \textsc{Tasaki}
on maximal antipodal subsets of symmetric $R$-spaces. \par
Since a maximal antipodal subset $\Antipodal$ of $X_I$ is the orbit of a Weyl group of $P,$ 
we see from the work of  \textsc{Berndt, Console} and \textsc{Fino} \cite[p.\ 86]{BCF}, where they refer to
a result of \textsc{Sanchez} \cite{Sanchez}, that the cardinality $|\Antipodal|$ of $\Antipodal$ 
has the following topological meaning
\begin{equation}\label{EQ: card max antipodal}
|\Antipodal|=\mathrm{dim}\, H_*(X_I,\Z_2).
\end{equation}
The maximal cardinality of an antipodal 
subset of a compact Riemannian symmetric space is called its \emph{$2$-number.} This invariant was introduced by 
\textsc{Chen} and \textsc{Nagano} in \cite{CN}.
Equation \eqref{EQ: card max antipodal} generalizes a result of \textsc{Takeuchi} \cite{Ta89} on the $2$-number of 
symmetric $R$-spaces.

\section{Maximality of $\Gamma^I$}\label{SECT: Maximality}

For a subset $J$ of $\{1,\dots, r\}$  we set $k_J:=\prod_{j\in J}k_j=\exp(\pi \xi_J)\in\LieK$
(see \eqref{EQ: k-j}) and
$$\gamma^J:=\prod_{j\in J}\gamma^j=\Int(k_J)\in\Aut(\LieK_0)$$
with the understanding that $k_\emptyset=e$ and $\xi_\emptyset=0.$
If $J\neq J',$ then $\gamma^J\neq \gamma^{J'}.$ Indeed, we may assume that $J\setminus J'\neq\emptyset.$
For $j\in J\setminus J'$ we see that the differential $\gamma^J_*$ of $\gamma^J$ at the identity 
satisfies $\gamma^J_*|_{\lieK_{\alpha_j}}=\Ad_\LieG(k_J)|_{\lieK_{\alpha_j}}
=-\mathrm{id}|_{\lieK_{\alpha_j}}$ while $\gamma^{J'}_*|_{\lieK_{\alpha_j}}=\Ad_\LieG(k_{J'})|_{\lieK_{\alpha_j}}
=\mathrm{id}|_{\lieK_{\alpha_j}}.$ We conclude that for any $\beta\in\Gamma^{I_\reg}$ there exists 
a unique subset $J_\beta\subset\{1,\dots, r\}$ such that 
$\beta=\gamma^{J_\beta}.$\par
Let $I$ be a non-empty subset of $I_\reg=\{1,\dots,r\}$ and let 
$\hat{\Gamma}$ be a subgroup of $\Gamma^{I_\reg}$ 
such that $(\LieK_0,\LieH_I,\hat{\Gamma})$ is a $\hat{\Gamma}$-symmetric triple.
By condition \eqref{EQ: Gamma sym triple} this implies  in particular
$$\fix(\hat{\Gamma}, \LieK_0)=\lieH_I,$$
where  the Lie algebra of  $\Fix(\hat\Gamma,\LieK_0)$ is
\begin{eqnarray*}
 \fix(\hat{\Gamma}, \LieK_0)&=& \{X\in\lieK\; |\; \forall t\in\R\;\; \forall \hat\gamma\in\hat\Gamma:
 \hat\gamma(\exp(tX))=\exp(tX)\}\\
 &=& \{X\in\lieK\; |\; \forall t\in\R\;\; \forall \hat\gamma\in\hat\Gamma:
 \Int(k_{J_{\hat{\gamma}}})(\exp(tX))=\exp(tX)\}\\
 &=& \{X\in\lieK\; |\; \forall t\in\R\;\; \forall \hat\gamma\in\hat\Gamma:
\exp(t\Ad_\LieG(k_{J_{\hat{\gamma}}})X)=\exp(tX)\}\\
 &=&\{X\in\lieK\; |\; \forall \hat\gamma\in\hat\Gamma:
 \Ad_\LieG(k_{J_{\hat{\gamma}}})X=X\}\\
 &=& \lieM\oplus \sum\limits_{\alpha\in\hat\Root^+}\lieK_\alpha\\
\end{eqnarray*}
with $\hat\Root^+=\{\alpha\in\Root^+\;|\; \forall \hat\gamma\in\hat\Gamma: 
\alpha(\xi_{J_{\hat{\gamma}}})\; \text{even}\}.$
Recalling that $\lieH_I=\lieM\oplus \sum\limits_{\alpha\in\Root^0_I}\lieK_\alpha$
we get 
\begin{equation}
\label{EQ: hatRoot+=Root-0-I}
 \hat\Root^+=\Root^0_I
\end{equation}
where $\Root^0_I$ is defined in \eqref{EQ: Root-0-I}. 
We claim that
\begin{equation}
\label{EQ: hatGamma subset Gamma-I}
 \hat\Gamma\subset\Gamma^I.
\end{equation}
Suppose  $\hat\Gamma\not\subset\Gamma^I$ and take 
$\hat\gamma\in\hat\Gamma\setminus \Gamma^I.$ By assumption 
$J_{\hat\gamma}\setminus I$ is non-empty. We take $s\in J_{\hat\gamma}\setminus I.$
Then the simple root $\alpha_s$ lies in $\Root^0_I$ since $s\notin I,$ but $\alpha_s\notin  \hat\Root^+$
since $\alpha_s(\xi_{J_{\hat\gamma}})=\sum\limits_{j\in J_{\hat\gamma}}\alpha_s(\xi_j)=1.$
This contradicts \eqref{EQ: hatRoot+=Root-0-I}.\par
Since  $\hat\Gamma$ is contained in $\Gamma^I$ and since 
$(\LieK_0,\LieH_I,\hat{\Gamma})$ is a $\hat{\Gamma}$-symmetric triple we have
$$\Fix(\Gamma^I,\LieK_0)_0\subset\Fix(\hat\Gamma,\LieK_0)_0\subset\LieH_I.$$ 
Summing up and recalling Observation \ref{OBS: right hand inclusion} we have shown: 

\begin{prop}
 Let $I$ be a non-empty subset of $I_\reg=\{1,\dots,r\}$ and let $\hat{\Gamma}$ be a subgroup of $\Gamma^{I_\reg}.$
If $(\LieK_0,\LieH_I,\hat{\Gamma})$ is a $\hat{\Gamma}$-symmetric triple, then $\hat{\Gamma}\subset \Gamma^I$ and 
$I$ is admissible.
\end{prop}

\begin{ex} We now give an example of an admissible subset $I$ 
and a \emph{proper} subgroup $\hat{\Gamma}$ of $\Gamma^I$
such that 
 $(\LieK_0,\LieH_I,\hat{\Gamma})$ is a $\hat{\Gamma}$-symmetric triple:\par
Let  $P$ be a bottom symmetric space whose root system $\Root$ is of type $A_r$ with $r\geq 3.$ 
Then
every non-empty subset of $\{1,\dots, r\}$ is admissible. The set of simple roots $\Sigma$ is isomorphic to 
$\{\alpha_1=e_1-e_2,\; \alpha_2=e_2-e_3,\; \dots, \; \alpha_r=e_r-e_{r+1}\}$
and the positive roots are of the form 
$e_j-e_k=\sum\limits_{j\leq s<k} \alpha_s$ with $1\leq j<k\leq r+1$ 
(see e.g.\ \cite[Planche I, p.\ 250]{Bourbaki}).\par
 Take $I=\{i_1,i_2,i_3\}$ with $1\leq i_1<i_2<i_3\leq r$ and
consider the subgroup $$\hat{\Gamma}:=\{e,\, \gamma^{i_1}\gamma^{i_3},\, \gamma^{i_2},\; \gamma^{i_1}\gamma^{i_2}\gamma^{i_3}\}
=\langle \gamma^{i_1}\gamma^{i_3},\, \gamma^{i_2}\rangle\cong 
(\Z_2)^2$$
of $\Gamma^I\cong 
(\Z_2)^3.$ 
We  show that $(\LieK_0,\LieH_I,\hat{\Gamma})$ is a $\hat{\Gamma}$-symmetric triple. 
 Since $\LieH_I\subset\Fix(\Gamma^I,\LieK_0)\subset \Fix(\hat\Gamma,\LieK_0)$ we only have to show
 $\Fix(\hat\Gamma,\LieK_0)_0\subset \LieH_I$ which amounts to 
 $\fix(\hat\Gamma,\LieK_0)\subset \lieH_I.$ By the calculations at the beginning of this section, it is sufficient to show that 
 $\hat\Root^+\subset\Root^0_I.$ Let 
 $\alpha=e_j-e_k=\sum\limits_{j\leq s<k} \alpha_s$ with $1\leq j<k\leq r+1$ be a positive root. 
 \begin{itemize}
  \item Assume that  $j\leq i_2<k.$ Then $\alpha(\xi_{i_2})=1$ and $\alpha\notin \hat\Root^+,$ because $\gamma^{i_2}\in\hat\Gamma.$
  \item Assume that $k\leq i_2.$ 
	    \begin{itemize}
		\item If $j\leq i_1<k\leq i_2,$ then $\alpha(\xi_{i_1})=1,\; \alpha(\xi_{i_3})=0$ and therefore 
		$\alpha(\xi_{i_1}+\xi_{i_3})=1.$ Since $\gamma^{i_1}\gamma^{i_3}\in\hat\Gamma,$ we have
		$\alpha\notin \hat\Root^+.$
		\item If $k\leq i_1,$ then $\alpha(\xi_{i_1})=\alpha(\xi_{i_2})=\alpha(\xi_{i_3})=0$ and 
		$\alpha\in\Root^0_I.$
		\item If $i_1<j<k\leq i_2,$ then $\alpha(\xi_{i_1})=\alpha(\xi_{i_2})=\alpha(\xi_{i_3})=0$ and 
		$\alpha\in\Root^0_I.$
	    \end{itemize}
  \item Assume that $i_2<j.$ 
	    \begin{itemize}
		\item If $i_2<j\leq i_3<k,$ then $\alpha(\xi_{i_1}+\xi_{i_3})=1.$ Since $\gamma^{i_1}\gamma^{i_3}\in\hat\Gamma,$ we have
		$\alpha\notin \hat\Root^+.$
		\item If $i_3<j,$ then  $\alpha(\xi_{i_1})=\alpha(\xi_{i_2})=\alpha(\xi_{i_3})=0$ and 
		$\alpha\in\Root^0_I.$
		\item If $i_2<j<k\leq i_3,$ then $\alpha(\xi_{i_1})=\alpha(\xi_{i_2})=\alpha(\xi_{i_3})=0$ and 
		$\alpha\in\Root^0_I.$
	    \end{itemize}
\end{itemize}
If $P$ is associated with the orthogonal symmetric Lie algebra $(\lieG=\lieSU_{r+1},\, \lieK=\lieSO_{r+1}),$  
then the coset space $\LieK_0/\LieH_I$ can be identified with the real flag manifold 
  $$\{(V_1,V_2,V_3)\; |\; \{0\}\subsetneq V_1\subsetneq V_2\subsetneq V_3 \subsetneq\R^{r+1},\;
 \dim(V_k)=i_k\;
\text{for}\;  k=1,2,3\}.$$
Almost all classical $\Z_2\times\Z_2$-symmetric spaces 
have been classified by \textsc{Bahturin} and \textsc{Goze} \cite{Bahturin-Goze}. The missing classical and all exceptional cases have 
been added by \textsc{Kollross} \cite{Kollross}.
\end{ex}


\section*{Acknowledgements}
The authors are grateful to \textsc{Jost-Hinrich Eschenburg} for helpful discussions. The second author 
would like to thank the University of Augsburg for 
their hospitality during  his visit.

\end{document}